\documentclass[12pt]{amsart}
\usepackage{graphicx} % Required for inserting images

\usepackage[english]{babel}     % include this always
\usepackage[utf8]{inputenc}
\usepackage{amsmath}
\usepackage{amsfonts}
\usepackage{amssymb}
\usepackage{mathrsfs}
\usepackage{stmaryrd}
\usepackage{hyperref}
\usepackage{xcolor}
\usepackage{comment}
\usepackage{enumerate}
\usepackage{tikz-cd}
\usepackage{pgfplots}
\usepackage{doi}
\usepackage{mathdots}
\usepackage{makecell}
\usepackage{tikz}
\usetikzlibrary{positioning,arrows.meta}
\usetikzlibrary{calc}
\tikzset{
    vertex/.style={circle, fill=black, inner sep=1.5pt},
    edge/.style={->, >=Stealth, semithick},
    inflabel/.style={midway, fill=white, inner sep=1pt}
}

\usepackage{todonotes}
\usepackage{comment}
\specialcomment{auxproof}
{\mbox{}\newline\textbf{BEGIN: AUX-PROOF}\dotfill\newline}
{\mbox{}\newline\textbf{END: AUX-PROOF}\dotfill\newline}
\excludecomment{auxproof}

\usepackage{amsthm}
\newtheorem{theorem}{Theorem}[section]

\newtheorem{lemma}[theorem]{Lemma}

\newtheorem{corollary}[theorem]{Corollary}

\theoremstyle{definition}
\newtheorem{definition}[theorem]{Definition}
\newtheorem{example}[theorem]{Example}
\newtheorem{remark}[theorem]{Remark}

\newcommand{\midbullet}{\mathbin{\raisebox{0.2ex}{\scalebox{0.7}{$\bullet$}}}}
\newcommand{\KK}{\mathrm{KK}}
\newcommand{\KKeq}{\approx_{\mathrm{KK}}}

\title{On split exact sequences and KK-equivalences of amplified graph C*-algebras}
\author{Jesse Reimann, Sophie Emma Zegers}
\date{}

\begin{document}

\maketitle
\begin{abstract}
    We give a general methodology for constructing split exact sequences of amplified graph C*-algebras with sinks. This in turn allows us to construct explicit KK-equivalences with $\mathbb{C}^N$ for a large class of C*-algebras, including the quantum Grassmannian $\mathrm{Gr}_q(2,4)$. We discuss compatibility with known (quantum) CW-constructions and give an explicit KK-equivalence between the classical and quantum projective spaces $\mathbb{C}P^1$ and $\mathbb{C}P_q^1$. 
\end{abstract}
\section{Introduction}

Graph C*-algebras form a class of C*-algebras that has been tremendously useful in the classification program. As their generating relations arise from a directed graph, their structural properties are likewise largely encoded in their underlying graphs. For example, the ideal structure~\cite{bates_hong_raeburn_szymanski_2002}, K-theory~\cite{drinen_tomforde_2002}, and classification theory~\cite{eilers_restorff_ruiz_sorensen_2021} are exceptionally well understood and can be expressed in terms of graph-theoretical properties. 

\vspace{0.3cm}
Special attention has been paid to the graph C*-algebras of \emph{amplified} graphs, i.e.\ graphs in which between any two vertices there are either zero or (countably) infinitely many edges. While perhaps somewhat counterintuitive at first, the amplified structure indeed further simplifies the description of algebraic properties in terms of graph properties, where finite emitters often require extra care. In~\cite{eilers_ruiz_sørensen_2012}, amplified graph C*-algebras were classified via a collection of graph moves, which laid the groundwork for the classification of all unital graph C*-algebras~\cite{eilers_restorff_ruiz_sorensen_2021}. Moreover, it is known that the underlying amplified graph can be recovered from the circle-equivariant K-theory of the graph C*-algebra~\cite{eilers_ruiz_sims_2022}.  

\vspace{0.3cm}
Beyond classification theory, graph C*-algebras have also found applications in noncommutative topology. A significant number of \emph{quantum spaces}, i.e.\ noncommutative generalisations of classical topological spaces, has been found to be isomorphic to an amplified graph C*-algebra. Examples include complex quantum projective spaces, described in~\cite{hong_szymanski_2002} as amplified graph C*-algebras and obtained as fixed point algebras of odd-dimensional Vaksman-Soibel'man quantum spheres $C(S_q^{2n+1})$~\cite{vaksman_soibelman_90} under the gauge action. A similar description has been found for quantum teardrops~\cite{brzezinski_szymanski_2018}, quantum flag manifolds~\cite{BKOS_qfm_prep,strung_qfm_mfo_22}, and quantum weighted projective spaces~\cite{brzezinski_szymanski_2018,kettner_satheesan_XX}.
More generally, in~\cite{d_andrea_zegers_2025} it was shown that graph C*-algebras of amplified acyclic graphs with finitely many vertices are isomorphic to fixed point algebras of Cuntz-Krieger algebras under the canonical gauge action. 

\vspace{0.3cm}
For complex quantum projective spaces, and more generally for quantum flag manifolds, it is known that they are KK-equivalent to their classical counterparts~\cite{neshveyev_tuset_2012}. Moreover, there exists such a KK-equivalence that is equivariant with respect to the action of the maximal torus~\cite{yamashita_2012}. However, explicit witnesses of these KK-equivalences are not known, which hinders our understanding of precisely which structures of classical flag manifolds are preserved under quantisation. Combined with the recent graph description of quantum flag manifolds, this motivates our search for KK-equivalences of amplified graph C*-algebras. Our work builds upon~\cite{arici_zegers_2023}, in which KK-equivalences of complex quantum projective spaces with classical spaces $\mathbb{C}^n$ were explicitly constructed from split exact sequences. More precisely, 
the graph C*-algebraic description of complex quantum projective spaces from~\cite{hong_szymanski_2002} allowed for the construction of an explicit splitting map $C(\mathbb{C}P_q^{n-1})\to C(\mathbb{C}P_q^n)$, which in turn yielded an explicit KK-equivalence between $C(\mathbb{C}P_q^n)$ and $C(\mathbb{C}P_q^{n-1})\oplus\mathbb{K}$. This explicit KK-equivalence was used to derive explicit generators of the K-theory of $C(\mathbb{C}P_q^n)$ from the representation theory of~$C(\mathbb{C}P_q^n)$ and the Vaksman-Soibel'man sphere $C(S_q^{2n+1})$.

\vspace{0.3cm}
In this work, we expand the work of~\cite{arici_zegers_2023} in the following ways.
\begin{enumerate}
    \item We give a large family of permitted splittings, which we expect to prove valuable in the search for explicit \emph{equivariant} KK-equivalences and explicit K-theory generators.
    \item Our construction applies to a large class of amplified graph C*-algebras, namely those with finitely many vertices and a sink. {In particular, our results apply to quantum spaces such as the quantum Grassmannian $Gr_q(2,4)$.}
\end{enumerate}
\vspace{0.3cm}
\paragraph{\bf Structure of the paper}
After introducing amplified graph C*-algebras in Section~\ref{sect: graph_CStar_prelims}, we discuss the construction of an (equivariant) KK-equivalence from a split exact sequence in Section~\ref{sect: KK}. Our main result, Theorem~\ref{thrm: expl_splitting}, is proven in Section~\ref{sect: main_result}. In Section~\ref{sect: applications}, we apply our main result to the quantum Grassmannian $Gr_q(2,4)$ and discuss connections with its CW-decomposition in~\cite{d_andrea_zegers_2025}.
\vspace{0.3cm}
\paragraph{\bf Notation}
We use $\mathbb{K}$ to refer to compact operators on a separable Hilbert space. If $E$ is a Hilbert C*-module, then $\mathbb{K}(E)$ will refer to the ``compact'' operators on $E$ in the sense of Definition~\ref{def: K_E}.  For a C*-algebra $B$, we let $M(B)$ denote the multiplier algebra of $B$.
\vspace{0.3cm}
\paragraph{\bf Acknowledgements} Part of this work was carried out during the first author’s visit to the University of Tokyo. JR thanks Yasuyuki Kawahigashi for his hospitality. We also thanks Enli Chen for helpful discussions.

\section{Graph C*-algebras of amplified directed graphs}\label{sect: graph_CStar_prelims}
We give a basic introduction to graph C*-algebras here, with a focus on amplified graph C*-algebras. For a thorough introduction we refer to~\cite{raeburn_2005} (though we note that the role of the source and range maps are interchanged in~\cite{raeburn_2005} and our work, as is common in much of the graph C*-algebra literature).
\begin{definition}[Graph C*-algebras]
    Let $\Gamma=(\Gamma^0,\Gamma^1,s,r)$ be a directed graph with vertices $\Gamma^0$, edges $\Gamma^1$, and source/range maps $s,r:\Gamma^1\to\Gamma^0$. The \emph{graph C*-algebra} $C^*(\Gamma)$ is the universal C*-algebra generated by $\{p_v,s_e\mid v\in\Gamma^0,\;e\in\Gamma^1\}$, where $(p_v)_{v\in\Gamma^0}$ are mutually orthogonal projections, $(s_e)_{e\in\Gamma^1}$ are partial isometries satisfying $s_e^*s_f=0$ for $e\neq f$, and the generators satisfy the \emph{Cuntz-Krieger relations}\newpage
    \begin{enumerate}
        \item[(CK1)] $s_e^*s_e=p_{r(e)}$, $e\in \Gamma^1$,
        \item[(CK2)] $s_es_e^*\le p_{s(e)}$, $e\in\Gamma^1$,
        \item[(CK3)] if $v\in\Gamma^0$ is a \emph{finite emitter}, i.e.\ such that $0<|\{e\in\Gamma^1\mid s^{-1}(e)=v\}|<\infty$, then~$p_v=\sum_{s^{-1}(e)=v}s_es_e^*$. 
    \end{enumerate}
\end{definition}

\tikzset{
    vertex/.style={circle, draw=black, fill=black, inner sep=1pt},
    edge/.style={->, >=Stealth}
}
\begin{figure}[h]
    \centering
    \begin{tikzpicture}[node distance=1.2cm and 2cm]

% 1) One vertex, no edges
\node[vertex] (g1v1) {};
\node[below=0.55cm of g1v1] {$\mathbb{C}$};

% 2) One vertex with loop
\node[vertex, right=1.5cm of g1v1] (g2v1) {};
\draw[edge] (g2v1) edge[loop above, in=45, out=135, min distance=10mm] (g2v1);
\node[below=0.5cm of g2v1] {$C(S^1)$};

% 3) Three vertices in a line
\node[vertex, right=1.5cm of g2v1] (g3v1) {};
\node[vertex, right=1cm of g3v1] (g3v2) {};
\node[vertex, right=1cm of g3v2] (g3v3) {};
\draw[edge] (g3v1) -- (g3v2);
\draw[edge] (g3v2) -- (g3v3);
\node[below=0.5cm of g3v2] {$M_3(\mathbb{C})$};

% 4) Infinite path
\node[vertex, right=1.5cm of g3v3] (g4v1) {};
\node[vertex, right=1cm of g4v1] (g4v2) {};
%\node[vertex, right=1cm of g4v2] (g4v3) {};
\node[right=0.8cm of g4v2] (g4dots) {$\dotsc$};
\draw[edge] (g4v1) -- (g4v2);
%\draw[edge] (g4v2) -- (g4v3);
\draw[edge] (g4v2) -- (g4dots);
\node[below=0.55cm of g4v2] {$\mathbb{K}$};

% 5) Two vertices with infinity label on edge
\node[vertex, right=1.5cm of g4dots] (g5v1) {};
\node[vertex, right=1cm of g5v1] (g5v2) {};
\draw[edge] (g5v1) -- node[midway, above] {$(\infty)$} (g5v2) coordinate[midway] (g5mid);
\node[below=0.55cm of g5mid] {$\tilde{\mathbb{K}}$};

\end{tikzpicture}
    \caption{Graph descriptions of selected C*-algebras. Here, $\tilde{\mathbb{K}}$ denotes the minimal unitisation of the compact operators, and $(\infty)$ denotes countably infinitely many edges from one vertex to another.}
    \label{fig:sample_graphs}
\end{figure}
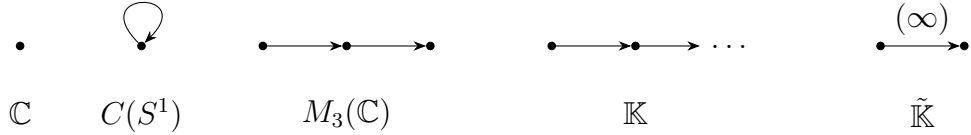

A selection of graph descriptions of familiar C*-algebras is given in Figure~\ref{fig:sample_graphs}. Within a graph $\Gamma$, a \emph{path}~$\alpha:=\alpha_1\cdots\alpha_n$ is a sequence of edges $\alpha_i\in\Gamma^1$ such that $r(\alpha_i)=s(\alpha_{i+1})$ for all $i=1,\dotsc,n-1$. Its source and range are defined as $s(\alpha):=s(\alpha_1)$, $r(\alpha):=r(\alpha_n)$. We will let a path with source $v$ and range $w$ be denoted by $v\to w$ when the edges of which the path consists are not relevant.
Moreover, graph C*-algebras carry a canonical \emph{gauge action} $\gamma: U(1)\to \mathrm{Aut}\;C^*(\Gamma)$, defined on generators as 
\begin{equation}\label{eqn: gauge_action}
    \gamma_z(p_v)=p_v,\quad \gamma_z(s_e)=zs_e.
\end{equation}

\vspace{0.3cm}
Many properties of graph C*-algebras can be read directly off the corresponding graph. For example, $C^*(\Gamma)$ is unital if and only if $|\Gamma^0|<\infty$, in which case $1_{C^*(\Gamma)}=\sum_{v\in \Gamma^0}p_v$. Moreover, one can gather information about the gauge-invariant ideal structure of $C^*(\Gamma)$ from the properties of $\Gamma$, see e.g.\ \cite{bates_hong_raeburn_szymanski_2002, eilers_ruiz_sørensen_2012}. 

\vspace{0.3cm}
A graph $\Gamma$ is said to be \emph{amplified} if for any two vertices $v,w\in\Gamma^0$, there are either no edges with source $v$ and range $w$, or (countably) infinitely many. In particular, (CK3) will be trivially satisfied for any family of projections and isometries if $\Gamma$ is amplified. Moreover, all ideals of amplified graph C*-algebras are gauge invariant~\cite{bates_hong_raeburn_szymanski_2002}. The ideal structure of amplified graphs is hence particularly well-understood and easily described in terms of subsets of $\Gamma^0$. A set $H\subseteq\Gamma^0$ is said to be \emph{hereditary} if for all $v\in H$ it holds that if there exists a path from $v$ to some $w\in\Gamma^0$, then $w\in H$.

\begin{theorem}[Ideal structure of amplified graphs, cf.\ {\cite[Section 3]{bates_hong_raeburn_szymanski_2002}}]
    Let $\Gamma$ be an amplified directed graph. Then there is a one-to-one correspondence between ideals $I\subseteq C^*(\Gamma)$ and hereditary\footnote{If $\Gamma$ is not amplified, the set $H$ is furthermore required to be \emph{saturated}, i.e.\ a non-sink vertex emitting finitely many edges must belong to $H$ if all its emitted edges have range in $H$. Moreover, if $\Gamma$ is not amplified, this construction only describes the {gauge-invariant} ideals of $C^*(\Gamma)$. See~\cite{bates_hong_raeburn_szymanski_2002}.} 
    subsets $H\subseteq \Gamma^0$. Moreover, $C^*(\Gamma)/I_H\simeq C^*(\Gamma\setminus H)$, where \begin{equation*}
        \Gamma\setminus H:=(\Gamma^0\setminus H, \Gamma^1\setminus\{e\in \Gamma^1\mid s(e)\in H\text{ or }r(e)\in H\}, s|_{(\Gamma\setminus H)^1}, r|_{(\Gamma\setminus H)^1}).
    \end{equation*}
\end{theorem}
By slight abuse of notation, we will refer to the vertices of $\Gamma$ and $\Gamma\setminus H$ by the same notation. If $v\in\Gamma^0$ is a sink, then $\{v\}\subset\Gamma^0$ is hereditary and $I_{\{v\}}\simeq\mathbb{C}$ if $v$ is also a source, and $I_{\{v\}}\simeq\mathbb{K}$ if $v$ is not a source~\cite{bates_hong_raeburn_szymanski_2002}.

Remarkably, if $\Gamma$ is an amplified graph with finitely many vertices, its structural properties depend on the presence or absence of \emph{paths} between vertices rather than adjacency. In fact, we may add or remove edges as long as the path structure is preserved.

\begin{theorem}[{\cite[Section 3]{eilers_ruiz_sørensen_2012}}]\label{thrm: add_edges}
    Let $\Gamma$ be an amplified graph with finitely many vertices. Let $v,w\in \Gamma^0$ be such that there is a path $\alpha$ in $\Gamma$ with $s(\alpha)=v$, $r(\alpha)=w$, $\ell(\alpha)\ge 2$, and such that there are no edges with source $v$ and range $w$. Let $\Tilde{\Gamma}$ be a directed graph with $\Tilde{\Gamma}^0=\Gamma^0$,  $\Tilde{\Gamma}^1=\Gamma^1\cup\{e_{v,w}^k\mid k\in\mathbb{N}\}$,
    and the source and range maps are the natural extensions of that on $\Gamma$ such that $s(e_{v,w}^k)=v,\;r(e_{v,w}^k)=w$ for all $k\in\mathbb{N}$.
    Then $C^*(\Gamma)\simeq C^*(\Tilde{\Gamma})$.
\end{theorem}

\section{KK-theory}\label{sect: KK}
We give a brief introduction to KK-theory, based on~\cite[Sections 13 and 17]{blackadar_1986}. As we are ultimately interested in the KK-theory of graph C*-algebras, we assume throughout this section all C*-algebras $A,B,...$ are separable and trivially graded. 
For a more general treatment, see~\cite{blackadar_1986,jensen_thomsen_1991}.
\begin{definition}[Hilbert C*-modules] Let $B$ be a C*-algebra. A \emph{Hilbert $B$-module} is a right $B$-module $E$ with a sesquilinear form $\langle\cdot,\cdot\rangle: E\times E\to B$, linear in the second argument, such that for all $x,y\in E$, $b\in B$ the following hold:\begin{enumerate}
    \item $\langle x,yb\rangle=\langle x,y\rangle b$,
    \item $\langle y,x\rangle=\langle x,y\rangle^*$,
    \item $\langle x,x\rangle\ge 0$ and $\langle x,x\rangle =0$ iff $x=0$, 
    \item $E$ is complete in the norm $\|x\|:=\|\langle x,x\rangle\|_B^{1/2}.$
\end{enumerate}
\end{definition}
A C*-algebra $B$ is a Hilbert $B$-module with the sesquilinear form $\langle a,b\rangle:=a^*b$. Moreover, any Hilbert space can be interpreted as a Hilbert $\mathbb{C}$-module. 

\vspace{0.3cm}
Let $\mathcal{L}(E)$ denote the \emph{adjointable} operators on $E$, i.e.\ module homomorphisms with a well-defined adjoint with respect to $\langle\cdot,\cdot\rangle$. Such operators are automatically bounded. In fact, if we consider a C*-algebra $B$ as a Hilbert $B$-module, then $\mathcal{L}(B)\simeq M(B)$.
In analogy to Hilbert spaces, one can define ``compact'' operators on a Hilbert C*-module as the closure of the analogues of rank one operators.
\begin{definition}[{$\mathbb{K}(E)$}]\label{def: K_E}
    Let $E$ be a Hilbert C*-module. For $x,y\in E$, define the operator $\theta_{x,y}:E\to E$, $\theta_{x,y}(z):=x\langle y,z\rangle$. Note $\theta_{x,y}\in\mathcal{L}(E)$. Define $\mathbb{K}(E)$ as the closure (with respect to the operator norm) of the linear span of $\{\theta_{x,y}\mid x,y\in E\}$. 
\end{definition}
\begin{remark}
    Note that elements of $\mathbb{K}(E)$ are not necessarily compact. For example, if $A$ is a unital C*-algebra, then $\mathrm{id}_A=\theta_{1,1}\in\mathbb{K}(A)$. However, when no confusion can arise, we will refer to elements of~$\mathbb{K}(E)$ as compact operators on $E$.
\end{remark}

We are now ready to give a definition of KK-classes.
There are multiple equivalent definitions; we will make use of the following two:
\begin{definition}[KK via Kasparov modules]
    Given two C*-algebras $A,B$, a \emph{Kasparov $(A,B)$-module} is a triple $(E,\phi,F)$, where \begin{itemize}
        \item $E$ is a Hilbert $B$-module,
        \item $\phi:A\to \mathcal{L}(E)$ is a *-homomorphism, 
        \item $F\in \mathcal{L}(E)$ is such that for all $a\in A$,  $[F,\phi(a)]$, $(F^2-1)\phi(a)$, and $(F^*-F)\phi(a)$ are elements of $\mathbb{K}(E)$. 
    \end{itemize}
    We define $\mathrm{KK}_h(A,B)$ to be the set of homotopy equivalence classes of Kasparov $(A,B)$-modules.
\end{definition}
\begin{definition}[KK via quasihomomorphisms] Given two C*-algebras $A,B$, a \emph{quasihomomorphism} from $A$ to $B$ is a pair $(\phi^+,\phi^-)$, where $\phi^{\pm}:A\to M(B\otimes\mathbb{K})$ satisfy $\phi^+(a)-\phi^-(a)\in B\otimes\mathbb{K}$ for all $a\in A$.

We define $\mathrm{KK}_c(A,B)$ as the set of homotopy equivalence classes of quasihomomorphisms from $A$ to $B$.
\end{definition}

If $A$ is separable and $B$ is $\sigma$-unital, then $\mathrm{KK}_h(A,B)\simeq \mathrm{KK}_c(A,B)$, see~\cite[Section 5.2]{jensen_thomsen_1991}. Since all graph C*-algebras are separable, they are in particular $\sigma$-unital. Thus these assumptions are always satisfied in our setting, and we will drop the subscript and use both pictures to describe elements of $\mathrm{KK}(A,B)$.

\vspace{0.3cm}
An important class of elements of $\mathrm{KK}(A,B)$ is arises from *-homomorphisms $\varphi:A\to B$. Such a *-homomorphism determines the Kasparov module $(B,\widetilde{\varphi},0)$ and the quasihomomorphism $(\widetilde{\varphi},0)$, where $\widetilde{\varphi}:A\to M(B)(\hookrightarrow M(B\otimes\mathbb{K}))$ is the natural extension of $\varphi$. 

\vspace{0.3cm}
The \emph{Kasparov product} $$\midbullet :\mathrm{KK}(A,B)\times \mathrm{KK}(B,C)\to \mathrm{KK}(A,C)$$ allows us to compose KK-classes. While calculating the Kasparov product of two elements explicitly can be quite technical, in some cases we can express it as a pushforward/pullback construction as follows (see~\cite[Section 4.3]{jensen_thomsen_1991} for details). Let $A,B,C$ be separable C*-algebras and let $\varphi:A\to B$ be a *-homomorphism, then \begin{equation*}
    \varphi^*:= [\varphi] \midbullet\;- : \mathrm{KK}(B,C)\to\mathrm{KK}(A,C),\quad \varphi^*[\phi^+,\phi^-] = [\phi^+ \circ\varphi, \phi^-\circ\varphi]
\end{equation*}
is well-defined.
Moreover, if $\varphi$ is quasi-unital\footnote{A *-homomorphism $\varphi:A\to B$ is \emph{quasi-unital} if there exists a projection $p\in M(B)$ such that $\overline{\varphi(A)B}=pB$.}, then there exists a well-defined strictly continuous\footnote{A *-homomorphism $\phi:M(A)\to M(B)$ is \emph{strictly continuous} if it is continuous w.r.t.\ the strict topologies on the multiplier algebras. The strict topology on $M(A)$ is generated by the seminorms $x\mapsto\|ax\|$ and $x\mapsto\|xa\|$, $a\in A$.} extension $\overline{\varphi}:M(\mathbb{K}\otimes A)\to M(\mathbb{K}\otimes B)$ and hence a well-defined map
\begin{equation*}
    \varphi_*:= -\;\midbullet [\varphi]: \mathrm{KK}(C,A)\to \mathrm{KK}(C,B),\quad \varphi_*[\phi^+,\phi^-]=[\overline{\varphi}\circ\phi^+, \overline{\varphi}\circ\phi^-].
\end{equation*}

Two C*-algebras $A$ and $B$ are called \emph{KK-equivalent}, denoted by $A\approx_{\mathrm{KK}} B$, if there exist $x\in \mathrm{KK}(A,B)$, $y\in \mathrm{KK}(B,A)$ such that $$x\midbullet y=1_A\in \mathrm{KK}(A,A),\quad y\midbullet x=1_B\in \mathrm{KK}(B,B).$$ If $A\approx_{\mathrm{KK}}B$, then for any separable C*-algebra $C$ it holds that $\KK(A,C)\simeq \KK(B,C)$ and $\KK(C,A)\simeq \KK(C,B)$. As the K-theory and K-homology groups of $A$ are isomorphic to $\KK(\mathbb{C},A)$ and $\KK(A,\mathbb{C})$, respectively, it follows that KK-equivalent C*-algebras have the same K-theory and K-homology. Moreover, satisfaction of the Universal Coefficient theorem is closely related to KK-equivalence with a commutative C*-algebra, see~\cite{rosenberg_schochet_1987,skandalis_1988}.  
\subsection{KK-equivalence from split exact sequences}\label{sect: blackadar}
Given a split exact sequence
\begin{equation}\label{eqn: ses}
\begin{tikzcd}
0 \arrow[r] & J \arrow[r, "\iota"] & A \arrow[r, "q", shift left] & B \arrow[l, "s", shift left] \arrow[r] & 0
\end{tikzcd}
\end{equation}
of separable C*-algebras, it is well known that $A\KKeq J\oplus B$, which can be shown abstractly for example by a six-term exact sequence argument. However, this KK-equivalence can also be constructed explicitly, as was shown in~\cite[Section 5.4]{arici_zegers_2023}. As it is central to our work, we briefly review the construction. 

\vspace{0.3cm}
The element $[\iota]\oplus[s]\in\mathrm{KK}(J\oplus B,A)$ has the following inverse with respect to the Kasparov product: Fix a minimal projection $e\in\mathbb{K}$ and let $e_A:A\to\mathbb{K}\otimes A$ be given by~$e_A(a):=e\otimes a$. Moreover, let $r_J:M(\mathbb{K}\otimes A)\to M(\mathbb{K}\otimes J)$ be the unique map such that 
\begin{equation}\label{eqn: define_rJ}
    (\mathrm{id}_{\mathbb{K}}\otimes \iota)(r_J(m)x)=m(\mathrm{id}_{\mathbb{K}}\otimes \iota)(x)
\end{equation} for all $m\in M(\mathbb{K}\otimes A)$ and all $x\in\mathbb{K}\otimes J$ (see~\cite[Ex.\ 1.1.9]{jensen_thomsen_1991} for existence and uniqueness of this map). Let $\omega:= r_J\circ e_A$ and $\pi:=(\omega, \omega\circ s\circ q)$, then \begin{equation*}
    [\pi]\oplus[q]=[(\omega, \omega\circ s\circ q)]\oplus[q]\in\mathrm{KK}(A,J\oplus B)
\end{equation*}
is an inverse to $[\iota]\oplus[s]\in\mathrm{KK}(J\oplus B,A)$, in the sense that \begin{align*}
    ([\iota]\oplus[s]) \midbullet   ([\pi]\oplus[q]) &= \iota_*(  [\pi]) + s_* \circ q_*([1_A])=[1_A]\in \mathrm{KK}(A,A), \\
     ( [\pi]\oplus[q])\midbullet ([\iota]\oplus[s]) &= (\iota^* + s^*)(  [\pi] \oplus q^*([1_B])) = [1_J] \oplus [1_B] \in \mathrm{KK}(J\oplus B, J\oplus B).
\end{align*}
\subsection{Equivariant KK-equivalences from split exact sequences}
In this section, we state an equivariant version of the KK-equivalence result of Section~\ref{sect: blackadar}. Although this result is presumably well-known to experts, we were unable to find it in the literature. Throughout this section, let $G$ be a compact group. The noncompact case is treated e.g.\ in~\cite{blackadar_1986}.

\vspace{0.3cm}
We first give a definition of equivariant KK-classes, where we restrict ourselves to the equivariant Kasparov picture. The following definitions are from~\cite[Section 20]{blackadar_1986}.
\begin{definition}[continuous $G$-action on Hilbert C*-modules]
    Let $B$ be a $G$-C*-algebra, i.e.\ a C*-algebra with a continuous $G$-action. Let $E$ be a Hilbert $B$-module with an action of $G$ on~$E$, i.e.\ a homomorphism from $G$ into the group of bounded invertible linear operators on $E$, such that~$g\cdot (xb)=(g\cdot x)(g\cdot b)$ for all $g\in G$, $x\in E$, $b\in B$. This action is \emph{continuous} if $g\mapsto\|\langle g\cdot x,g\cdot x\rangle\|$ is continuous for all $x\in E$.
\end{definition}
\begin{definition}[Kasparov $G$-modules, $\mathrm{KK}^G$] Let $A,B$ be separable graded {$G$-C*-algebras}. A \emph{Kasparov~$G$-module} for $A,B$ is a triple $(E,\phi,F)$, where\begin{itemize}
    \item $E$ is a countably generated Hilbert $B$-module with a (SOT-)continuous action of $G$,
    \item $\phi:A\to \mathcal{L}(E)$ is an equivariant graded *-homomorphism,
    \item $F\in \mathcal{L}(E)$ is $G$-continuous (i.e.\ $g\mapsto g\cdot F$ is norm continuous), of degree 1, and such that $[F,\phi(a)]$, $(F^2-1)\phi(a)$, $(F^*-F)\phi(a)$, and~${(g\cdot F-F)\phi(a)}$ are all elements of $\mathbb{K}(E)$ for all $a\in A$, $g\in G$.
\end{itemize}
 We define $\mathrm{KK}^G(A,B)$ as the set of homotopy equivalence classes of Kasparov $G$-modules for~$A,B$.   
\end{definition}

Given an equivariant *-homomorphism $\varphi:A\to B$, the Kasparov $G$-module $(B,\varphi,0)$ gives rise to an element in $\mathrm{KK}^G(A,B)$.
Given a split exact sequence as in~\eqref{eqn: ses},
with $\iota, s,q$ $G$-equivariant and $J,A,B$ $G$-C*-algebras, we will now show that the construction of Section~\ref{sect: blackadar} yields a $G$-equivariant KK-equivalence. As  $\iota,q,s$ are $G$-equivariant *-homomorphisms, the Kasparov modules obtained from them automatically yield corresponding $\mathrm{KK}^G$-classes.   We therefore restrict our attention to showing that $[\pi]=[(\omega,\omega\circ s\circ q)]\in\mathrm{KK}^G(A,J)$. To verify equivariance, it is convenient to pass to the Kasparov picture, where this class is represented by the Kasparov module\begin{equation*}
    \left(J\oplus J^{\mathrm{op}}, \widetilde{\omega} \oplus\widetilde{\omega}\circ s\circ q, \left[\begin{matrix}
        0 & 1 \\
        1 & 0
    \end{matrix}\right]\right),
\end{equation*}
where $\widetilde{\omega}$ is the canonical embedding of $A$ into $M(J)\simeq\mathcal{L}(J)$ such that \begin{equation}\label{eqn: define_tildeomega}
    \iota(\widetilde{\omega}(a)x)=a\iota(x)
\end{equation} for all $a\in A$, $x\in J$, see~\cite[Proposition 3.12.8]{pedersen1979c} and~\cite[Example 17.1.2]{blackadar_1986}.
\begin{lemma}\label{lem: omega_tile_equivariant}
    If $\iota$ is $G$-equivariant, then $\widetilde{\omega}$ is $G$-equivariant.
\end{lemma}
\begin{proof}
    Let $\alpha_g$, $\rho_g$ denote the action of $G$ on $A,J$, respectively. The canonical action of $G$ on $\mathcal{L}(J)$ is given by $(g\cdot T)(x):=(\rho_g\circ T\circ \rho_{g^{-1}})(x)$.
    Since~\eqref{eqn: define_tildeomega} uniquely determines $\widetilde{\omega}$, it suffices to show that \begin{equation*}
    \iota([g\cdot\widetilde{\omega}(a)]x)=(\alpha_ga)\iota(x)
    \end{equation*}
    for all $x\in J$ and $g\in G$. Indeed, we have \begin{equation*}
        \iota([g\cdot\widetilde{\omega}(a)]x) = \iota([\rho_g\circ\widetilde{\omega}(a)]\rho_{g^{-1}}(x))=(\alpha_g\circ \iota)(\widetilde{\omega}(a)\rho_{g^{-1}}(x))=\alpha_g(a(\iota\circ \rho_{g^{-1}}(x))),
    \end{equation*} where we used~\eqref{eqn: define_tildeomega} in the last equality. Using that $\alpha_g$ is a *-homomorphism, we obtain \begin{equation*}
        \alpha_g(a(\iota\circ \rho_{g^{-1}}(x))) = \alpha_g(a)(\alpha_g\circ \iota\circ \rho_{g^{-1}}(x)) = \alpha_g(a)(\iota\circ \rho_g\circ \rho_{g^{-1}}(x))=\alpha_g(a)\iota(x).
    \end{equation*}
\end{proof}

\begin{lemma}\label{lem: splitting_equivariant_G_module}
    $ \left(J\oplus J^{\mathrm{op}}, \widetilde{\omega} \oplus\widetilde{\omega}\circ s\circ q, \left[\begin{matrix}
        0 & 1 \\
        1 & 0
    \end{matrix}\right]\right)$ is a Kasparov $G$-module. 
\end{lemma}
\begin{proof}
As $J$ is a $G$-algebra, it follows that $J\oplus J^{\mathrm{op}}$ is a Hilbert $J$-module with a continuous $G$-action. By Lemma~\ref{lem: omega_tile_equivariant}, $\phi:=\widetilde{\omega}\oplus\widetilde{\omega}\circ s\circ q$ is equivariant.
Let $F:= \left[\begin{matrix}
        0 & 1 \\
        1 & 0
    \end{matrix}\right]$.
   As in the non-equivariant case, it follows directly that $[F,\phi(a)],$ $(F^*-F)\phi(a)$, $(F^2-1)\phi(a)$  are compact for all $a\in A$.  Let $\rho$ denote the action of $G$ on $J$. Then $F$ is $G$-continuous. Since for $x_1,x_2\in J$, we have  \begin{align*}
            &\left(g\cdot \left[\begin{matrix}
            0 & 1 \\ 1 & 0
        \end{matrix}\right]\right)(x_1\oplus x_2) = \left(\left[\begin{matrix}
            \rho_g & 0 \\ 0 & \rho_g
        \end{matrix}\right]\circ\left[\begin{matrix}
            0 & 1 \\ 1 & 0
        \end{matrix}\right]\circ\left[\begin{matrix}
            \rho_{g^{-1}} & 0 \\ 0 & \rho_{g^{-1}}
        \end{matrix}\right]\right)(x_1\oplus x_2) =x_2\oplus x_1,
        \end{align*}
        i.e.\ $F$ is indeed $G$-invariant. Thus $(g\cdot F-F)\phi(a)$ is trivially compact.
\end{proof}
Hence $[\iota]\oplus[s]\in \mathrm{KK}^{G}(J\oplus B,A)$ is an equivariant KK-equivalence with inverse given by~$[\pi]\oplus[q]\in \mathrm{KK}^G(A,J\oplus B)$.

\section{Explicit splittings for amplified graph C*-algebras}\label{sect: main_result}
In this section we construct an explicit splitting $\sigma: C^*(\Gamma\setminus\{v_s\})\to C^*(\Gamma)$ for the short exact sequence
\begin{equation}\label{eqn: ses_minus_one_sink}
\begin{tikzcd}
0 \arrow[r] & \mathbb{K} \arrow[r, "\iota"] & C^*(\Gamma) \arrow[r, "q" ] & C^*(\Gamma\setminus\{v_s\}) 
\arrow[r] & 0,
\end{tikzcd}
\end{equation}
which (cf.\ Section~\ref{sect: graph_CStar_prelims}) arises from taking the quotient of $C^*(\Gamma)$ by its ideal generated by the projection $p_{v_s}$, where $v_s$ is a sink vertex. In fact, we construct a large family of unital splittings, as described in the following theorem.
\begin{theorem}\label{thrm: expl_splitting}
    Let $\Gamma$ be an amplified graph with finitely many vertices. Let $v_s\in \Gamma^0$ be a sink. Let $v_*\in\Gamma^0$ be such that either $v_*$ is a source or such that for each $w\in \Gamma^0$ with edges $w\to v_*$ there exists a path in $\Gamma$ from $w$ to $v_s$. If there are edges with source $v$ and range $w$, let $(s_{v,w}^i)_{i\in\mathbb{N}}$ denote the partial isometries associated to these edges. Then the map $\sigma:C^*(\Gamma\setminus\{v_s\})\to C^*(\Gamma)$, defined on generators as \begin{align*}
        p_v&\mapsto p_v,&v\in\Gamma^0\setminus\{v_s,v_*\}, \\
        p_{v_*}&\mapsto p_{v_*}+p_{v_s},& \\
        s^i_{v,w}&\mapsto s^i_{v,w},&v\in \Gamma^0,\;w\in \Gamma^0\setminus\{v_s,v_*\}, \\
        s^i_{v,v_*}&\mapsto s^i_{v,v_*}+ s^i_{v,v_s}, &v\in \Gamma^0\setminus\{v_s\},
    \end{align*}
    defines a unital splitting for the short exact sequence~\eqref{eqn: ses_minus_one_sink}.
\end{theorem}
\begin{proof}
    First assume $v_*$ is not a source. If $v$ is a vertex with edges $v\to v_*$, we may by Theorem~\ref{thrm: add_edges} add edges $v\to v_s$ such that the partial isometries $s_{v,v_s}^i$, $i\in\mathbb{N}$, exist.
    The fact that the splitting is a well-defined unital *-homomorphism follows as in~\cite{arici_zegers_2023}; we include it for the convenience of the reader. By construction, $\{\sigma(p_v)\mid v\in \Gamma^0\}$ are mutually orthogonal projections. Moreover, for every $v_1,v_2\in \Gamma^0$, $w_1,w_2\in\Gamma^0\setminus{v_s}$, $i,j\in\mathbb{N}$ we have $$\sigma(s^i_{v_1,w_1})^* \sigma(s_{v_2,w_2}^j)=\delta_{i,j}\delta_{v_1,v_2}\delta_{w_1,w_2}p_{w_1}=\delta_{i,j}\delta_{v_1,v_2}\delta_{w_1,w_2}\sigma(p_{w_1}),$$ where we note that for $w_1=w_2=v_*$ it follows from the Cuntz-Krieger relations of the generators of $C^*(\Gamma)$ that \begin{align*}
        \sigma(s^i_{v_1,v_*})^* \sigma(s_{v_2,v_*}^j)&= (s^i_{v_1,v_*}+ s^i_{v_1,v_s})^* (s^j_{v_2,v_*}+ s^j_{v_2,v_s}) \\&= \delta_{i,j}\delta_{v_1,v_2} ((s^i_{v_1,v_*})^* s^i_{v_1,v_*} +  (s^i_{v_1,v_s})^*s^i_{v_1,v_s}) \\
        &=\delta_{i,j}\delta_{v_1,v_2} (p_{v_*}+p_{v_s})\\
        &=\delta_{i,j}\delta_{v_1,v_2}\sigma (p_{v_*}).
    \end{align*}

    In a similar manner, we obtain $\sigma(s^i_{v,w}) \sigma(s_{v,w}^i)^*\le \sigma(p_{v})$, which follows immediately in case~$w\neq v_*$. If $w=v_*$, we use \begin{equation*}
         s^i_{v,v_*}(s^i_{v,v_s})^* =  s^i_{v,v_*}(s^i_{v,v_*})^*s^i_{v,v_*}(s^i_{v,v_s})^*s^i_{v,v_s}(s^i_{v,v_s})^*= s^i_{v,v_*}{p_{v_*}p_{v_s}}(s^i_{v,v_s})^* =0
    \end{equation*}
    and obtain \begin{equation*}
        \sigma(s^i_{v,v_*}) \sigma(s_{v,v_*}^i)^* = (s^i_{v,v_*}+ s^i_{v,v_s})(s^i_{v,v_*}+ s^i_{v,v_s})^* = \underbrace{s^i_{v,v_*}(s^i_{v,v_*})^*}_{\le p_{v_*}} 
        + \underbrace{s^i_{v,v_s}(s^i_{v,v_s})^*}_{\le p_{v_s}} \le \sigma(p_{v_*}).
    \end{equation*}
    Finally, unitality follows from the explicit form of the unit in graph C*-algebras, as
    \begin{equation*}
    \sigma(1_{C^*(\Gamma\setminus\{v_s\})})=\sigma(\sum_{v\in\Gamma^0\setminus\{v_s\}}p_v)=\sum_{v\in\Gamma^0\setminus\{v_s,v_*\}}p_v+p_{v_*}+p_{v_s}=1_{C^*(\Gamma)}.
    \end{equation*}
     If $v_*$ is a source, note that $\sigma$ maps the generators of $C^*(\Gamma\setminus\{s_v\})$ (except for~$p_{v_*}$) to their natural counterparts in $C^*(\Gamma)$, which is again a unital *-homomorphism by the same arguments. Finally, note that in either case~$q\circ \sigma=\mathrm{id}_{C^*(\Gamma\setminus\{v_s\})}$ holds.
\end{proof}

\begin{example}
    Consider the the amplified graph with structure\begin{equation*}
\tikzset{
    vertex/.style={circle, draw=black, fill=black, inner sep=1pt},
    edge/.style={->, >=Stealth}
}
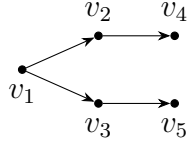
       \begin{tikzpicture}[node distance=0.8cm and 0.9cm]
    \node[vertex, label=below:{$v_1$}] (v1) {};
    \node[vertex, right=of v1, yshift=0.45cm, label=above:{$v_2$}] (v2) {};
    \node[vertex, right=of v1, yshift=-0.45cm, label=below:{$v_3$}] (v3) {};
    \node[vertex, right=of v2, label=above:{$v_4$}] (v4) {};
    \node[vertex, right=of v3, label=below:{$v_5$}] (v5) {};

    \draw[edge] (v1) -- (v2);
    \draw[edge] (v1) -- (v3);
    \draw[edge] (v2) -- (v4);
    \draw[edge] (v3) -- (v5);
    \end{tikzpicture}
    \end{equation*}
and let $v_s:=v_4$ be the designated vertex to be removed. Then $\{v_1,v_2,v_3\}$ are all valid choices for $v_*$. In particular, note that if $v_*=v_3$, then each edge with range $v_3$ has source~$v_1$, and there exists a path $v_1\to v_4$. On the other hand, $v_5$ is not a valid choice of $v_*$, as no path $v_3\to v_4$ exists.
\end{example}
\begin{remark}
    The natural embedding $\iota: C^*(\Gamma\setminus\{v_s\})\to C^*(\Gamma)$, defined on generators as \begin{align*}
        p_v&\mapsto p_v,\;v\in\Gamma^0\setminus\{v_s\}, \qquad
        s^i_{v,w}\mapsto s^i_{v,w},\;v\in \Gamma^0,\;w\in \Gamma^0\setminus\{v_s\},
    \end{align*}
    is a splitting for~\eqref{eqn: ses_minus_one_sink} as well, albeit not a unital one.
\end{remark}

\begin{corollary}
    Let $v_{s_1},\dotsc,v_{s_k}$ be sinks in $\Gamma$. Let $$\sigma_j: C^*(\Gamma\setminus\{v_1,\dotsc,v_j\})\to C^*(\Gamma\setminus\{v_1,\dotsc,v_{j-1}\}),\; j=1,\dotsc,k$$ denote splittings obtained as in Theorem~\ref{thrm: expl_splitting}. Then $\sigma:=\sigma_1\circ\dotsc\circ\sigma_k$ is a unital splitting for the short exact sequence
    \begin{equation*}\label{eqn: ses_minus_k_sinks}
\begin{tikzcd}
0 \arrow[r] & \mathbb{K}^k \arrow[r, "\iota"] & C^*(\Gamma) \arrow[r, "q"] & C^*(\Gamma\setminus\{v_{s_1},\dotsc,v_{s_n}\}) 
\arrow[r] & 0,
\end{tikzcd}
\end{equation*}
\end{corollary}
\begin{proof}
    Follows by the same proof as Theorem~\ref{thrm: expl_splitting}, using that $\mathbb{K}^k$ is isomorphic to the ideal generated by~$p_{v_{s_1}},\dotsc,p_{v_{s_k}}$ in $C^*(\Gamma)$, see~\cite[Proposition~3.2]{brzezinski_szymanski_2018}
\end{proof}

\begin{corollary}\label{cor: KK_equiv_graphs}
    If $\Gamma$ is an acyclic amplified graph with finitely many vertices, then a KK-equivalence $C^*(\Gamma)\approx_{\mathrm{KK}}\mathbb{C}^{|\Gamma^0|}$ can be constructed explicitly.
\end{corollary}
\begin{proof}
    Let $N:=|\Gamma^0|$. Repeatedly applying Theorem~\ref{thrm: expl_splitting} yields a collection of split exact sequences
    \begin{equation*}\label{eqn: ses_minus_ith_sink}
\begin{tikzcd}
0 \arrow[r] & \mathbb{K} \arrow[r, "\iota_i"] & C^*(\Gamma\setminus\{v_1,\dotsc,v_{i-1}\}) \arrow[r, "q_i"] & C^*(\Gamma\setminus\{v_1,\dotsc,v_{i}\}) 
\arrow[r] & 0,
\end{tikzcd}
\end{equation*}
    with splittings $\sigma_i: C^*(\Gamma\setminus\{v_1,\dotsc,v_{i}\})\to C^*(\Gamma\setminus\{v_1,\dotsc,v_{i-1}\})$ constructed as in Theorem~\ref{thrm: expl_splitting}, where $v_i$ denotes the chosen sink vertex in $(\Gamma\setminus\{v_1,\dotsc,v_{i-1}\})^0\subset\Gamma^0$ and the choice of the designated vertex $v_*$ at each step is arbitrary.  

    \vspace{0.3cm}
    Define $\Gamma_0:=\Gamma$, $\Gamma_i:= \Gamma_{i-1}\setminus\{v_i\}=\Gamma\setminus\{v_1,\dotsc,v_i\}$.  Following~\ref{sect: blackadar}, we obtain a chain of KK-equivalences \begin{equation*}
        C^*(\Gamma)\approx_{\mathrm{KK}} \mathbb{K}\oplus C^*(\Gamma_1)\approx_{\mathrm{KK}}\mathbb{K}^2\oplus C^*(\Gamma_2)\approx_{\mathrm{KK}}\dotsc\approx_{\mathrm{KK}}\mathbb{K}^{N-1}\oplus C^*(\Gamma_{N-1}).
    \end{equation*}

    \vspace{0.3cm}
    Following~\cite[Theorem 6.2]{arici_zegers_2023}, an explicit KK-equivalence is given by ${\Pi_\Gamma\in \mathrm{KK}(C^*(\Gamma),\mathbb{K}^{N-1}\oplus\mathbb{C})}$ and $I_\Gamma\in \mathrm{KK}(\mathbb{K}^{N-1}\oplus\mathbb{C},C^*(\Gamma))$, which (using the notation from Section~\ref{sect: blackadar}) are defined as\begin{align*}
        \Pi_\Gamma&:= [\pi_1]\oplus ([q_{1}]\midbullet{C^*(\Gamma_{1})}[\pi_2])\oplus ([q_2\circ q_1]\midbullet{C^*(\Gamma_2)}[\pi_3])\oplus \dotsc \\
        &\qquad \oplus([q_{N-2}\circ\dotsc\circ q_2\circ q_1]\midbullet{C^*(\Gamma_{N-2})}[\pi_{N-1}]) \oplus [q_{N-1}\circ\dotsc\circ q_2\circ q_1], \\
        I_{\Gamma} &:= [\iota_1] \oplus [s_1\circ \iota_2] \oplus [s_1 \circ s_2 \circ \iota_3] \oplus \dots \\
        &\qquad \oplus [s_{1}\circ s_{2}\circ \dotsc\circ s_{N-2}\circ \iota_{N-1}] \oplus [s_1\circ s_2\circ\dotsc\circ s_{N-2}\circ s_{N-1}].
    \end{align*}
       By~\cite[Section 5.6.3]{debord_lescure_2008}, there moreover exists an explicit KK-equivalence $\mathbb{K}\approx_{\mathrm{KK}}\mathbb{C}$, given by $[\varphi]\in \mathrm{KK}(\mathbb{C},\mathbb{K})$, where $\varphi:\mathbb{C}\ni 1\mapsto e\in \mathbb{K}$ for some rank one projection $e$, and $[(H,\iota,0)]\in\mathrm{KK}(\mathbb{K},\mathbb{C)}$, where $H$ is a separable Hilbert space and $\iota$ the natural action of the compact operators on that Hilbert space. Altogether, we hence obtain an explicit KK-equivalence $C^*(\Gamma)\approx_{\mathrm{KK}}\mathbb{C}^N$.
\end{proof}
\begin{corollary} If $\Gamma$ is an acyclic amplified graph with $|\Gamma^0|<\infty$, then we obtain an explicit equivariant KK-equivalence $C^*(\Gamma)\approx_{\mathrm{KK}^{U(1)}}\mathbb{K}^{|\Gamma^0|-1}\oplus\mathbb{C}$ .
\end{corollary}
\begin{proof}
    Note that all splittings in Theorem~\ref{thrm: expl_splitting} are equivariant *-homomorphisms with regard to the gauge action~\eqref{eqn: gauge_action}. Recalling~\ref{lem: splitting_equivariant_G_module}, all KK-equivalences from Theorem~\ref{thrm: expl_splitting} in the proof of Corollary~\ref{cor: KK_equiv_graphs} are also $\mathrm{KK}^{U(1)}$-equivalences.
\end{proof}
\begin{remark}
    Note that $\mathbb{K}\not\approx_{\mathrm{KK}^{U(1)}}\mathbb{C}$ and hence $C^*(\Gamma)\not\approx_{\mathrm{KK}^{U(1)}}\mathbb{C}^{|\Gamma^0|}$. This can be seen from their graph description (see Section~\ref{sect: graph_CStar_prelims}), as the gauge action is trivial on $\mathbb{C}$ but not on~$\mathbb{K}$.
\end{remark}
     
\section{Application: KK-equivalences and CW-decomposition of $Gr_q(2,4)$}\label{sect: applications}

Classically, flag manifolds are homogeneous spaces, arising as quotients of complex simply connected semisimple Lie groups by their parabolic subgroups\footnote{Subgroups containing the Borel subgroup.}.  Examples include complex projective spaces $\mathbb{C}P^{n-1}\simeq SU(n)/S(U(1)\times U(n-1))$ and  Grassmannians $Gr(k,n)\simeq SU(n)/S(U(k)\times U(n-k))$, which can be interpreted as the space of one-dimensional and $k$-dimensional subspaces of $\mathbb{C}^n$, respectively. Via a quantisation procedure, one can obtain
\emph{quantum flag manifolds}~\cite{stokman_dijkhuizen_1999}.  Arising similarly as fixed point algebras of quantised Lie groups under coactions of corresponding quantum subgroups, they form an important class of quantum homogeneous spaces. 
More recently, they have been described as amplified graph C*-algebras~\cite{BKOS_qfm_prep,strung_qfm_mfo_22} (see also~\cite{d_andrea_zegers_2025}). We briefly review the construction here, specialised to A-series quantum flag manifolds which include the examples we will study. For details regarding classical Lie-theoretic concepts we refer to~\cite{procesi2006lie}.

\vspace{0.3cm}
Quantum (and classical) flag manifolds are classified by their tagged \emph{Dynkin diagrams}, where a subset of the nodes 
$s_1,\dotsc,s_n$ of the Dynkin diagram $A_n$ has been tagged, see Figure~\ref{fig: qfm_graph_example}. From the Dynkin diagram, one can recover the \emph{Weyl group}, which captures the reflection symmetries of the root system of the underlying Lie algebra. For A-series Dynkin diagrams, the associated Weyl group is the symmetric group $S_{n+1}$, i.e.\ the group with generators $s_1,\dotsc,s_n$ such that \begin{enumerate}
    \item $s_i^2=1$, $i=1,\dotsc,n$,
    \item $s_is_js_i=s_js_is_j$ if $|i-j|=1$,
    \item $s_i$ and $s_j$ commute if $|i-j|\neq 1$.
\end{enumerate}
It is well known that $W=S_n$ is finite, and that each word $w\in W$ has a unique reduced word representation, i.e.\ $w=s_{i_n}\cdots s_{i_1}$, $s_{i_1},\dotsc,s_{i_n}\in \{s_1,\dotsc,s_n\}$, and there exists no such representation with fewer letters. We let $\ell(w):=n$ denote the \emph{length} of $w$. For the identity~$e\in W$, we set $\ell(e):=0$.

\vspace{0.3cm}
Let $S$ denote the set of untagged nodes of $A_n$, and $W_S$ the subgroup of $W$ generated by~$S$. The graph underlying the associated quantum flag manifold is constructed as follows. \begin{itemize}
    \item vertices $\Gamma^0=\{\text{reduced words in }W/W_S\}$
    \item edges $\Gamma^1=\{e^k_{v,w}\mid v,w\in \Gamma^0 \text{ such that }v\text{ is a subword of }w\text{ and }\ell(w)=\ell(v)+1.\}$
\end{itemize}
An equivalent construction $\Gamma^0$ is to set $\Gamma^0$ to be the set consisting of the identity and all reduced words in $W$ that begin in a letter corresponding to a marked node in the Dynkin diagram (read from right to left). We give an example of the construction of the graphs for $\mathbb{C}P_q^3$ and $Gr_q(2,4)$ in Figure~\ref{fig: qfm_graph_example}.

\begin{figure}[h]

\setlength{\tabcolsep}{12pt} % Default value: 6pt
\renewcommand{\arraystretch}{3} % Default value: 1
    \centering
    \begin{tabular}{>{\centering\arraybackslash}m{3cm}cl}
       \makecell{(Quantum) flag \\ manifold}& \makecell{$\mathbb{C}P_q^3$}  & \hspace*{1.6cm}\tikz[baseline=(X.base)] \node (X) {$Gr_q(2,4)$}; \\
        \makecell{Dynkin\\ diagram} & 
        
        \begin{tikzpicture}[
        baseline=(current bounding box.center),
    node distance=0.8cm,
    dynkin node/.style={
        circle,
        draw=black,
        fill=white,
        thick,
        minimum size=5mm,
        inner sep=0pt,
    },
    xnode/.style={
        minimum size=5mm,
        inner sep=0pt
    },
    every path/.style={thick}
]
% First diagram
\node[xnode, label=below:{$s_1$}] (x1) at (0,0) {};
\node[dynkin node, label=below:{$s_2$}] (a2) [right=of x1] {};
\node[dynkin node, label=below:{$s_3$}] (a3) [right=of a2] {};

\draw (x1) -- (a2) -- (a3);
\draw[line width=1.2pt] ($(x1.center)+(-0.2,-0.2)$) -- ($(x1.center)+(0.2,0.2)$);
\draw[line width=1.2pt] ($(x1.center)+(-0.2,0.2)$) -- ($(x1.center)+(0.2,-0.2)$);

\end{tikzpicture} & 
\hspace*{0.75cm}{\begin{tikzpicture}[
baseline=(current bounding box.center),
    node distance=0.8cm,
    dynkin node/.style={
        circle,
        draw=black,
        fill=white,
        thick,
        minimum size=5mm,
        inner sep=0pt
    },
    xnode/.style={
        minimum size=5mm,
        inner sep=0pt
    },
    every path/.style={thick}
]

% Second diagram
\node[xnode, label=below:{$s_2$}] (x2) at (0,0) {};
\node[dynkin node, label=below:{$s_1$}] (b1) [left=of x2] {};
\node[dynkin node, label=below:{$s_3$}] (b3) [right=of x2] {};

\draw (b1) -- (x2) -- (b3);
\draw[line width=1.2pt] ($(x2.center)+(-0.2,-0.2)$) -- ($(x2.center)+(0.2,0.2)$);
\draw[line width=1.2pt] ($(x2.center)+(-0.2,0.2)$) -- ($(x2.center)+(0.2,-0.2)$);

\end{tikzpicture}}  \\
      
      \makecell{\; \\ graph}  & 
\begin{tikzpicture}[
baseline=(current bounding box.center),
node distance=0.9cm]
\path[use as bounding box] (-0.4,-0.9) rectangle (3.8,0.9);
    \node[vertex, label=below:{$e$}] (v1) {};
    \node[vertex, right=of v1, label=below:{$s_1$}] (v2) {};
    \node[vertex, right=of v2, label=below:{$s_2s_1$}] (v3) {};
    \node[vertex, right=of v3, label=below:{$s_3s_2s_1$}] (v4) {};

    \draw[edge] (v1) -- (v2);
    \draw[edge] (v2) -- (v3);
    \draw[edge] (v3) --  (v4);
\end{tikzpicture} & 
{\begin{tikzpicture}[
baseline=(current bounding box.center),
node distance=1.1cm and 0.9cm]
    \node[vertex, label=above:{$s_1s_2$}] (v3) at (0,0.45) {};
    \node[vertex, label=below:{$s_3s_2$}] (v4) at (0,-0.45) {};
    \node[vertex, label=below:{$s_2$}] (v2) at (-1.1,0) {};
    \node[vertex, label=below:{$e$}] (v1) at (-2.2,0) {};
    \node[vertex, label=below:{$s_1s_3s_2$}] (v5) at (1.1,0) {};
    \node[vertex, label=below:{$s_2s_1s_3s_2$}] (v6) at (2.7,0) {};

    \draw[edge] (v1) --  (v2);
    \draw[edge] (v2) --  (v3);
    \draw[edge] (v2) --  (v4);
    \draw[edge] (v3) -- (v5);
    \draw[edge] (v4) -- (v5);
    \draw[edge] (v5) --  (v6);
\end{tikzpicture}} \\
    \end{tabular}
    \caption{Construction of the graph description of $\mathbb{C}P_q^3$ and $Gr_q(2,4)$ from their tagged Dynkin diagrams. Here, the tagged nodes are denoted by a cross and the untagged nodes by a circle. The graphs are understood to be amplified; this has been suppressed from the notation to improve readability. Their vertices are labelled by the corresponding element in $W/W_S$.}
    \label{fig: qfm_graph_example}
\end{figure}
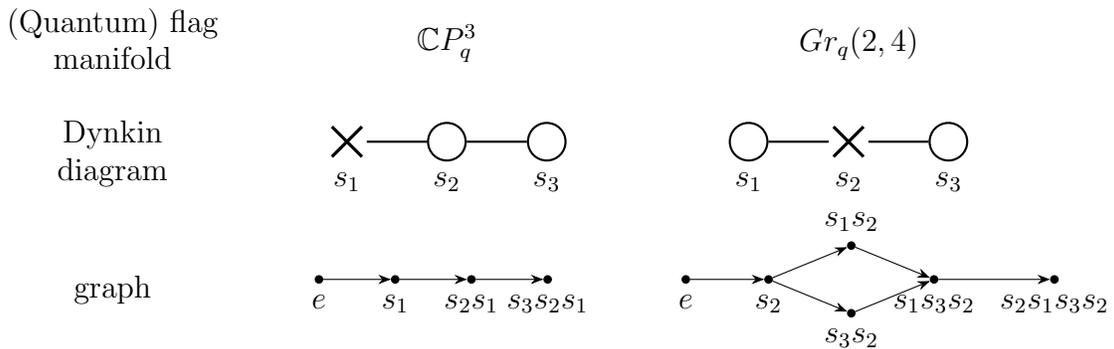

Classical flag manifolds are finite \emph{CW complexes}, i.e.\ they can be constructed by ``gluing'' a finite amount of open balls together (\cite[Section I.10.6.8]{procesi2006lie}, see also~\cite{hatcher2002algebraic} for an introduction to CW complexes). They may be defined inductively as a sequence of skeleta $X^0\subset X^1\subset\dotsc\subset X^N$, where $X^0$ is a discrete set and the inductive construction is given by pushout diagrams 

\[\begin{tikzcd}
	{\sqcup_kS^n} & {\sqcup_kD^{n+1}} \\
	{X^n} & {X^{n+1}}
	\arrow["{\sqcup_k\iota}", from=1-1, to=1-2]
	\arrow["f"', from=1-1, to=2-1]
	\arrow[from=1-2, to=2-2]
	\arrow[from=2-1, to=2-2]
\end{tikzcd}\]
where $f$ describes how the open balls $D^{n+1}$ are glued to $X^n$ to form $X^{n+1}$. For flag manifolds, we can describe this construction concretely in terms of the Weyl group and its quotient $W/W_S$ as follows.\begin{itemize}
    \item Each element of $W/W_S$ of length $k$ corresponds to a $2k$-cell, i.e.\ a topological space homeomorphic to an open ball in $\mathbb{R}^d$.
    \item For each such cell there exists a gluing map, where a cell corresponding to $w\in W_s$ is glued to the union of the boundaries of the cells corresponding to $u\in W_S$ such that $w=s_{i_k}u$ for some generator $s_{i_k}$ of $W$ and $\ell(w)=\ell(u)+1$.
\end{itemize}
An explicit description of the gluing map can be found in~\cite{procesi2006lie}. To give an example of this construction, the Weyl group quotient $W/W_S$ associated with $\mathbb{C}P^1$ is given by $\{e, s_1\}$ (compare to Figure~\ref{fig: qfm_graph_example}). Hence, the associated CW complex constructed as above consists of an open disc in $\mathbb{R}^2\simeq \mathbb{C}$, with its entire boundary glued to a single point. We hence recover that $\mathbb{C}P^1$ is isomorphic to Riemann sphere. 

\vspace{0.3cm}
Clearly, this construction is equivalent to constructing the graph description of quantum flag manifolds. In fact, in~\cite{d_andrea_zegers_2025}, a CW-decomposition of $Gr_q(2,4)$ is given, where the skeletons again have a description as an amplified graph C*-algebra, see Figure~\ref{fig:subgraphs_grassmannian}. In the noncommutative setting, a CW-decomposition is given by a series of pullback diagrams \[\begin{tikzcd}
	{C(X_q^n)} & {C(X_q^{n+1})} \\
	{\mathcal{B}} & {\mathcal{D}}
	\arrow["{\sqcup_k\iota}", from=1-2, to=1-1]
	\arrow["f"', from=1-1, to=2-1]
	\arrow[from=1-2, to=2-2]
	\arrow[from=2-2, to=2-1]
\end{tikzcd}\]
where $C(X_q^n)$ are C*-algebraic analogues of the skeleta in the classical case, and $\mathcal{B}$ and $\mathcal{D}$ are C*-algebras satisfying certain conditions~\cite{d_andrea_hajac_maszczyk_sheu_zielinski_2020}; in particular, their K-theory should agree with that of the classical odd-dimensional open balls (resp. even-dimensional spheres).

\vspace{0.3cm}
In~\cite{d_andrea_zegers_2025}, a quantum CW-decomposition of $C(Gr_q(2,4))$ is given, with an exact sequence of skeleta \begin{equation*}
    C(Gr_q(2,4))\to C(X_q^6)\to C(\mathbb{C}P_q^2\sqcup_{\mathbb{C}P_q^1}\mathbb{C}P_q^2)\to C(\mathbb{C}P_q^1)\to \mathbb{C}.
\end{equation*}
Moreover, all C*-algebras in this sequence have a description as amplified graph C*-algebras. See Figure~\ref{fig:subgraphs_grassmannian} for the graphs corresponding to $C(X_q^6)$ and $C(C(\mathbb{C}P_q^2\sqcup_{\mathbb{C}P_q^1}\mathbb{C}P_q^2))$.

\begin{figure}
    \centering
    \tikzset{
    vertex/.style={circle, fill=black, inner sep=1.5pt},
    edge/.style={->, >=Stealth, semithick},
    inflabel/.style={midway, fill=white, inner sep=1pt}
}
    \begin{tabular}{ccc}
       \begin{tikzpicture}[node distance=0.8cm and 0.9cm]
    \node[vertex] (v1) {};
    \node[vertex, right=of v1] (v2) {};
    \node[vertex, right=of v2, yshift=0.45cm] (v3) {};
    \node[vertex, right=of v2, yshift=-0.45cm] (v4) {};
    \node[vertex, right=of v3, yshift=-0.45cm] (v5) {};

    \draw[edge] (v1) --   (v2);
    \draw[edge] (v2) --   (v3);
    \draw[edge] (v2) --   (v4);
    \draw[edge] (v3) --   (v5);
    \draw[edge] (v4) --  (v5);
\end{tikzpicture}  &\qquad \qquad&\begin{tikzpicture}[node distance=0.8cm and 0.9cm]
    \node[vertex] (v1) {};
    \node[vertex, right=of v1] (v2) {};
    \node[vertex, right=of v2, yshift=0.45cm] (v3) {};
    \node[vertex, right=of v2, yshift=-0.45cm] (v4) {};

    \draw[edge] (v1) --   (v2);
    \draw[edge] (v2) --   (v3);
    \draw[edge] (v2) --   (v4);
\end{tikzpicture} \\
\qquad&&\\
       $X_q^6$  && $\mathbb{C}P_q^2\sqcup_{\mathbb{C}P_q^1}\mathbb{C}P_q^2$
       
    \end{tabular}
\caption{Amplified graphs appearing in the CW-decomposition of $C(Gr_q(2,4))$, yielding graph C*-algebras isomorphic to the skeleta $C(X_q^6)$ and $C(\mathbb{C}P_q^2\sqcup_{\mathbb{C}P_q^1}\mathbb{C}P_q^2)$. The infinite multiplicity of the edges has been suppressed from the notation.}
    \label{fig:subgraphs_grassmannian}
\end{figure}
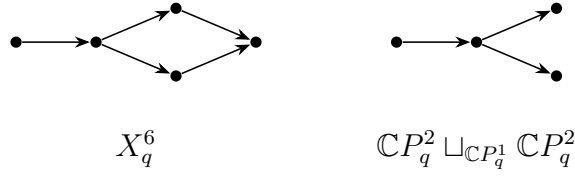

Our result reflects this CW-structure.

\begin{corollary}
    We have an explicit chain of KK-equivalences \begin{align*}
        C(Gr_q(2,4))&\KKeq \mathbb{K}\oplus C(X_q^6)\\
        &\KKeq \mathbb{K}^2\oplus C(\mathbb{C}P_q^2\sqcup_{\mathbb{C}P_q^1}\mathbb{C}P_q^2)\\
        &\KKeq\mathbb{K}^4\oplus C(\mathbb{C}P_q^1)\\ 
        &\KKeq \mathbb{K}^5\oplus\mathbb{C}\\
        &\KKeq \mathbb{C}^6.
    \end{align*}
\end{corollary}

Note that~\cite[19.9.3]{blackadar_1986} (see also~\cite[Section 5]{kasparov_1981}) outlines the construction of an explicit KK-equivalence $\mathbb{C}\KKeq C_0(\mathbb{R}^{2n})$ for all $n\in\mathbb{N}$ in terms of Dirac operators. 
Hence we can construct $\mathbb{K}\KKeq  C_0(\mathbb{D}^{2n-1})\simeq C_0(\mathbb{R}^{2n})$ explicitly. This is a step towards constructing an explicit witness for the KK-equivalence $C_q(G/P)\KKeq C(G/P)$~\cite{neshveyev_tuset_2012} (see also~\cite{yamashita_2012} for an equivariant version)  for general quantum flag manifolds, of which complex projective spaces and $Gr_q(2,4)$ are special cases.  

\vspace{0.3cm}
In fact, in the special case of $\mathbb{C}P_q^1$, the  short exact sequence obtained from the CW-decomposition \begin{equation}\label{eqn: CP1_SES}
\begin{tikzcd}
0 \arrow[r] &  C_0(\mathbb{D}^{1}) \arrow[r] & C(\mathbb{C}P^1) \arrow[r] & C(\{*\}) 
\arrow[r] & 0
\end{tikzcd}
\end{equation}
is split exact, with splitting $\mathrm{id}_{\{*\}}\mapsto id_{\mathbb{C}P^1}$, hence we have an explicit KK-equivalence $$C(\mathbb{C}P^1)\KKeq C_0(\mathbb{D}^{1})\oplus C(\{*\}).$$Altogether, we obtain a chain of explicit KK-equivalences \begin{equation*}
    C(\mathbb{C}P_q^1)\KKeq \mathbb{K}\oplus\mathbb{C}\KKeq C_0(\mathbb{D}^{1})\oplus\mathbb{C}\simeq C_0(\mathbb{D}^{1})\oplus C(\{*\})\KKeq C(\mathbb{C}P^1).
\end{equation*}

\vspace{0.3cm}
However, we do not expect this methodology to extend to all flag manifolds, as analogues of~\eqref{eqn: CP1_SES} do not split. As an example, consider $\mathbb{C}P^2$, which (as a real manifold) consists of one 0-cell, one 2-cell, and one 4-cell. The following sequence, obtained by restricting functions to the complement of the 4-cell, is exact.\begin{equation*}
    \label{eqn: CP2_SES}
\begin{tikzcd}
0 \arrow[r] &  C_0(\mathbb{D}^{3}) \arrow[r] & C(\mathbb{C}P^2) \arrow[r, "q"] & C(\mathbb{C}P^1)
\arrow[r] & 0
\end{tikzcd}
\end{equation*}
However, it is not split exact. Indeed, assume that there exists a *-homomorphism $s:C(\mathbb{C}P^1)\to C(\mathbb{C}P^2)$ such that $q\circ s=\mathrm{id}_{C(\mathbb{C}P^1)}$. Under the Gelfand transform, this is equivalent to there being a retract $r:\mathbb{C}P^2\to \mathbb{C}P^1$, i.e.\ $r$ is continuous and $r|_{\mathbb{C}P^1}=\mathrm{id}$ when viewing $\mathbb{C}P^1$ as a subspace of $\mathbb{C}P^2$. By functoriality, this implies that $\pi_k(r)\pi_k(\iota)=\mathrm{id}_{\pi_k(\mathbb{C}P^1)}$, where $\iota: \mathbb{C}P^1\to \mathbb{C}P^2$ is the natural embedding and $\pi_k$ denotes the $k$th \emph{homotopy group} of a topological space. In particular, $\pi_k(\iota)$ is injective for all $k\in\mathbb{N}$. However, it is known 
that $\pi_{3}(\mathbb{C}P^1)\simeq \mathbb{Z}$ whereas $\pi_3(\mathbb{C}P^2)\simeq \{0\}$, hence no such map can exist. See also~\cite{d_andrea_hajac_maszczyk_sheu_zielinski_2020} for a discussion of a similar result for $n\ge 3$.

\bibliographystyle{acm}
\bibliography{bibliography}

\end{document}